\documentclass{gtmon_a}
\pdfoutput=1

\proceedingstitle{Proceedings of the Nishida Fest (Kinosaki 2003)}
\conferencestart{28 July 2003}
\conferenceend{8 August 2003}
\conferencename{International Conference in Homotopy Theory}
\conferencelocation{Kinosaki, Japan}

\editor{Matthew Ando}
\givenname{Matthew}
\surname{Ando}

\editor{Norihiko Minami}
\givenname{Norihiko}
\surname{Minami}

\editor{Jack Morava}
\givenname{Jack}
\surname{Morava}

\editor{W Stephen Wilson}
\givenname{W Stephen}
\surname{Wilson}

\title{Determination of the multiplicative nilpotency\\of 
self-homotopy sets}

\author{Ken-ichi Maruyama}
\givenname{Ken-ichi}
\surname{Maruyama}
\address{Faculty of Education\\
Chiba University\\\newline
Chiba 263-8522\\
Japan}
\email{maruyama@faculty.chiba-u.jp}
\urladdr{}

\volumenumber{10}
\issuenumber{}
\publicationyear{2007}
\papernumber{16}
\startpage{281}
\endpage{292}

\doi{}
\MR{}
\Zbl{}

\arxivreference{}


\received{13 August 2004}
\revised{22 June 2005}
\accepted{}
\published{29 January 2007}
\publishedonline{29 January 2007}
\proposed{}
\seconded{}
\corresponding{}
\version{}

\makeatletter
\def\cnewtheorem#1[#2]#3{\newtheorem{#1}{#3}[section]
\expandafter\let\csname c@#1\endcsname\c@prop}

\let\xysavmatrix\xymatrix
\def\xymatrix{\disablesubscriptcorrection\xysavmatrix}
\AtBeginDocument{\let\bar\wbar\let\tilde\wtilde\let\hat\what}

\newtheorem{prop}{Proposition}[section]
\cnewtheorem{lemma}[prop]{Lemma}
\cnewtheorem{theorem}[prop]{Theorem}
\cnewtheorem{corollary}[prop]{Corollary}
\cnewtheorem{conjecture}[prop]{Conjecture}

\theoremstyle{remark}
\cnewtheorem{example}[prop]{Example}
\cnewtheorem{definition}[prop]{Definition}
\newtheorem*{rem}{Remark}

\makeatother  

\numberwithin{equation}{section}

\begin{document}

\begin{abstract}
The semigroup of the homotopy classes of the self-homotopy maps of a
finite complex which induce the trivial homomorphism on homotopy groups is
nilpotent.  We determine the nilpotency of these semigroups of compact Lie
groups and finite Hopf spaces of rank 2.  We also study the nilpotency of
semigroups for Lie groups of higher rank. Especially, we give Lie groups
with the nilpotency of the semigroups arbitrarily large.
\end{abstract}

\maketitle

\section*{Introduction}

Let $[X, Y]$ denote the based homotopy classes of maps from $X$ to $Y$. 
When $X = Y$, 
the self-homotopy set $[X, X]$ is a monoid by the binary operation induced 
by 
composition of maps. In this
paper we deal with a subset  ${\mathcal Z}^n(X)$  which consists of 
elements
inducing the trivial homomorphism on homotopy groups in dimensions 
$\leq n$, 
where $n$ is a natural number or $\infty$. ${\mathcal Z}^n(X)$ is a  
multiplicative 
subset of $[X, X]$, though it has no unit element and is merely a 
semigroup in general.
With respect to this binary operation the (multiplicative) 
nilpotency is defined (see \fullref{sec:sec1}).   
For a finite complex $X$, it is known by Arkowitz--Maruyama--Stanley
\cite{ams} that ${\mathcal Z}^n(X)$
is nilpotent for a sufficiently large integer $n$ and the nilpotency is 
bounded above by other 
invariants such as the cone length or the killing length of $X$. On the 
other hand,
lower bounds for the nilpotency of ${\mathcal Z}^n(X)$ or more desirably 
the precise value 
of it have not been studied except for a few cases, see \cite{ams}.
Our purposes are to know the nilpotency for compact Lie groups or finite
$H$--spaces of low rank and to obtain lower bounds of the nilpotency for 
more
general cases. 
We will determine the nilpotency of ${\mathcal Z}^n(X)$ when $X$ is a  
finite 1--connected
$H$--space
of rank 2. We will also determine 
the nilpotency of ${\mathcal Z}^n(X)$ in the rank 3 cases where
$X =SU(4)$ or $Sp(3)$. Incidentally
the nilpotencies are equal to 2 in these cases. However this is not the 
case for
Lie groups of higher rank.  Actually, we will
give 
the lower bounds of the nilpotency for $SU(n)$ or $Sp(n)$ and show that 
the nilpotency
could be arbitrarily large for these spaces.

We briefly review the sections. In \fullref{sec:sec1} we recall some basic 
definition and 
the work of Arkowitz--Maruyama--Stanley
\cite{ams} and Maruyama \cite{ma}.
In \fullref{sec:sec2} we first compute the nilpotency of ${\mathcal Z}^*(X)$ 
when  $X$ is a 1--connected compact Lie group of rank 2,
then apply the result to the case where $X$ is a 1--connected finite 
$H$--space of rank 2.
In \fullref{sec:sec3} we find the nilpotency of $SU(4)$ and $Sp(3)$.
To this end we use S Oka's work on the structures of self-homotopy sets of
$SU(4)$ and $Sp(3)$. In \fullref{sec:sec4} we derive a property of the 
nilpotency of
the rationalization of an $H$--space which is the key to the proof of the 
theorem on the nilpotency
of $SU(n)$ and $Sp(n)$ mentioned above.

I would like to thank Martin Arkowitz for helpful comments. I am also 
grateful to the referee for many useful remarks.

The author was partially supported by Grant-in-Aid for Scientific Research 
(14540063),
Japan Society for the Promotion of Science.
\section{Preliminaries}\label{sec:sec1}

In this section we fix our notation and recall some results in \cite{ams}. 
For spaces $X$  and $Y$, let ${\mathcal Z}^n(X, Y)$ denote the subset of 
$[X, Y]$ consisting 
of all homotopy classes $\alpha \in [X, Y]$ such that $\alpha_* = 0 \co 
\pi_i(X) \to \pi_i(Y)$
for $i \leq n$.
If $n = \infty$ we write ${\mathcal Z}^\infty(X, Y)$. We also write  
${\mathcal Z}(X, Y)$
for ${\mathcal Z}^{\dim X}(X, Y)$ if $n = \dim X$.
Finally we write ${\mathcal Z}^n(X)$ for ${\mathcal Z}^n(X, X)$ and
${\mathcal Z}^\infty(X)$ for ${\mathcal Z}^\infty(X, X)$.

${\mathcal Z}^n(X)$ is a semigroup by the binary operation induced by 
composition of maps.
\begin{definition}
If there exists an integer $t \geq 1$ such that $a_1 \circ a_2 \circ 
\cdots \circ a_t = 0$
for all $a_1, a_2, \ldots, a_t \in {\mathcal Z}(X)$,
then ${\mathcal Z}(X)$ is called nilpotent.
The smallest such $t$ is called the nilpotency of ${\mathcal Z}(X)$ and 
written  $t(X)$.
Similarly we define the nilpotency of ${\mathcal Z}^\infty(X)$ and denote 
it by $t_\infty(X)$.
\end{definition}
Clearly $t_\infty(X) \leq t(X)$.  In \cite{ams} it is shown that if $X$ is 
a finite complex,
then  $t(X)$ and thus $t_\infty(X)$ are finite and the following 
inequalities allow us
to know about the upper bounds for the nilpotency. 
\begin{theorem}[Arkowitz, Maruyama and Stanley~\cite{ams}]
If $X$ is a 1--connected finite complex then,
$$t_\infty(X) \leq t(X) \leq {\rm kl}_s(X) \leq {\rm cl}_s(X)$$ 
where ${\rm kl}_s(X)$ is the spherical killing length of $X$ and ${\rm 
cl}_s(X)$
is the spherical cone length of $X$.
\end{theorem}
In this paper we deal with the spaces which have multiplications. When
$G$ is a group-like finite complex, $[G, G]$ is a nilpotent group 
(see Whitehead \cite{wh}) and 
${\mathcal Z}^n(G)$ is a subgroup. There exists the following   
naturality property of localization which will be used in the proofs of 
the results in 
the later sections.  
\begin{prop}[Maruyama~\cite{ma}] \label{localization} Let $G$ be a
group-like finite complex. Then 
with respect to the group structures induced from the multiplication of 
$G$,  
${\mathcal Z}(G)_p \cong {\mathcal Z}(G_p)$ and ${\mathcal Z}^\infty(G)_p 
\cong
{\mathcal Z}^\infty(G_p)$  
for any prime $p$. Here $X_p$ is the localization of $X$ at $p$. 
\end{prop}
\section{The rank 2 case} \label{sec:sec2}
In this section we consider simply connected compact Lie groups of rank 2 
and related 
$H$--spaces.
\begin{theorem}\label{rank2}
Let $G$ be a 1--connected compact Lie group of rank 2. Then 
$$ t(G) = t_\infty(G) = 2.$$
\end{theorem} 
\begin{proof}
It is known that $G$ is isomorphic to one of the  Lie groups 
$$ S^3 \times S^3, SU(3), Sp(2), G_2.$$ 
We have the exact sequence
$$0 \to [S^3 \wedge S^3, S^3 \times S^3] \xrightarrow{q^*} [S^3 \times 
S^3, S^3 \times S^3] \to
[S^3 \vee S^3, S^3 \times S^3] \to 0,$$
where $q  \co S^3 \times S^3 \to S^{3} \wedge S^{3}$ is the projection map to 
the smash product. 
Since generally the projection 
$q \co S^m \times S^n \to S^{m} \wedge S^{n} \simeq S^{m + n}$
belongs to 
${\mathcal Z}^\infty(S^m \times S^n, S^{m + n})$ and by the above exact 
sequence  
we easily obtain 
$${\mathcal Z}^\infty(S^3 \times S^3) = {\mathcal Z}(S^3 \times S^3)
= {\rm Im} \; q^* = \pi_6(S^3) \oplus \pi_6(S^3) = \mathbb{Z}_{12} \oplus 
\mathbb{Z}_{12}.$$
Let  $f_1, f_2 \in $  
${\mathcal Z}^\infty(S^3 \times S^3)$
( = ${\mathcal Z}(S^3 \times S^3)$), 
then $f_1 \circ f_2 = 0$.  
We already know that ${\mathcal Z}^\infty(S^3 \times S^3)$ is not trivial. 
Thus 
$t_\infty(S^3 \times S^3) = t(S^3 \times S^3) = 2.$ Now we turn to the 
other cases.
Let $G$ be one of our Lie groups other than $S^3 \times S^3$ and let
$$q_G^* \co \pi_{\dim G}(G) \to {\mathcal Z}(G)$$ 
denote the induced map of $q_G \co G \to S^{\dim G}$, the pinching map to 
the top cell.
If $G = SU(3)$ or $Sp(2)$, then  ${\mathcal Z}^\infty(G) = {\mathcal 
Z}(G) \not= 0$ by Maruyama \cite{ma2}
(isomorphic to $\mathbb{Z}_{12}, \mathbb{Z}_{120}$ respectively)
and they coincide with ${\rm Im}~q_G^*$. Let $f_1, f_2 \in $  
${\rm Im}~q_G^*$ and $f_1 = q_G^*(x_1),\ f_2 = q_G^*(x_2)$, then $f_1 
\circ f_2$ is trivial since in the composition 
$$f_1 \circ f_2 \co G \xrightarrow{q_G}  S^{\dim G} \xrightarrow{x_2} G 
\xrightarrow{q_{G}} S^{\dim G}
\xrightarrow{x_1} G$$
$x_2$ is of finite order while $[S^{\dim G}, S^{\dim G}]$
is isomorphic to $\mathbb{Z}.$ 
Therefore we obtain that $ t(G) = t_\infty(G) = 2$
for $G = SU(3)$ and $Sp(2)$. Though it is shown that ${\mathcal Z}(G_2)$
is isomorphic to $\mathbb{Z}_2 \oplus  \mathbb{Z}_2 \oplus \mathbb{Z}_8 \oplus 
\mathbb{Z}_{21}$,
${\mathcal Z}^\infty(G_2)$ is not determined (for some partial results 
see \cite{ma2}). However by \cite{ma2} and {\=O}shima \cite{os}
${\mathcal Z}(G_2)$
is generated by the elements of ${\rm Im}\;q_{G_2}^*$ and $[1, \alpha]$.
Here $[1, \alpha]$ is the commutator of the identity map and some $\alpha 
\in [G_2, G_2]$
of infinite order. $[1, \alpha]$ is known to be of order 2. 
Thus ${\mathcal Z}^\infty(G_2)$
is not trivial as $[1, \alpha]$ is an element of ${\mathcal 
Z}^\infty(G_2).$
Let $f = q_{G_2}^*(x)
\in {\rm Im}\;q_{G_2}^*  (x \in \pi_{14}(G_2))$
and $g \in {\mathcal Z}(G_2)$,
then $f \circ g = 0 = g \circ f $. For, 
$$ g \circ f \co G_2 \xrightarrow{q_{G_2}} S^{14} \xrightarrow{x} G_2 
\xrightarrow{g} G_2$$
is trivial since $g \in {\mathcal Z}(G_2)$. 
$$ f \circ g \co G_2 \xrightarrow{g} G_2  \xrightarrow{q_{G_2}} S^{14} 
\xrightarrow{x} G_2$$
is also trivial as for the previous cases.
Moreover $[1, \alpha]\circ h = [h, \alpha \circ h] = 0$ for any
$h \in {\mathcal Z}(G_2)$, because $\alpha \circ h$ is an element of
${\mathcal Z}(G_2)$
and the group ${\mathcal Z}(G_2)$
is commutative. Let $f, f' \in {\rm Im}\;q_{G_2}^*.$ By the above arguments
$$(f + [1, \alpha])\circ (f' + [1, \alpha]) = f \circ (f' + [1, \alpha])
+ [1, \alpha] \circ (f' + [1, \alpha]) = 0.$$
Thus we have shown that all the compositions of the elements of ${\mathcal 
Z}(G_2)$
are trivial.  
As was noted above ${\mathcal Z}^\infty(G_2)$ contains an nontrivial 
element
$[1, \alpha]$, $t_\infty(G_2) > 1$ (actually $|{\mathcal Z}^\infty(G_2)|$
is greater than 42 see \cite{ma2}). Therefore $t(G_2) = t_\infty(G_2) = 
2$,
and we obtain the result.
\end{proof}
The assertion of \fullref{rank2} also holds for finite $H$--spaces of rank 
2. 
\begin{theorem}\label{rank2h}
Let $X$ be a 1--connected finite $H$--space of rank 2. Then $$ t(X) = 
t_\infty(X) = 2.$$
\end{theorem}
We will use the following lemma.
\begin{lemma}\label{upperlemma}
Let $X$ be a finite nilpotent space. If $t(X_p) \leq n$ for all prime 
numbers $p$, then $t(X)
\leq n$. The same is true for $t_\infty(X)$. 
\end{lemma}
\begin{proof}
Let $a_1, \ldots , a_n$ be elements of ${\mathcal Z}(X)$. Then 
$$(a_1 \circ \cdots \circ a_n)_p = 0$$
for any prime number $p$ by the assumption. Thus  $a_1 \circ \cdots \circ 
a_n$ is trivial by Hilton--Mislin--Roitberg
\cite[Corollary 5.12, Chapter II]{hmr}.
\end{proof}

\begin{proof}[Proof of \fullref{rank2h}]
By the classical result of Mimura, Nishida and Toda~ \cite{mnt}, a
1--connected finite $H$--space $X$ of rank 2 is
homotopy equivalent to one of $S^3 \times S^3$, $SU(3)$, $E_k$ ($k$ = 0, 
1, 3, 4, 5), 
$S^7 \times S^7$ or $G_{2, b}$ ($-2 \leq b \leq 5$). Here $E_k$ is the 
principal
$S^3$--bundle over $S^7$ with the characteristic class $k\omega \in 
\pi_7(BS^3)$, $\omega$ 
a generator, and  $G_{2, b}$ is the principal $S^3$--bundle over the 
Stiefel manifold $V_{7, 2}$ 
induced by a suitable map $f_b \co V_{7, 2} \to BS^3$. We note that $E_1$ 
=$Sp(2)$, $G_{2, 0}
= G_2$ and  
we have already shown the assertion for $SU(3)$, $Sp(2)$ and $G_2$ in 
\fullref{rank2}.
By definition we obtain
\begin{align*}
(E_3)_p &\simeq Sp(2)_p \text{ for } p \not= 3 &
(E_3)_3 &\simeq S^3_3 \times S^7_3 \\
(E_4)_p &\simeq Sp(2)_p\text{ for }p \not= 2 &
(E_4)_2 &\simeq S^3_2 \times S^7_2\\
(E_5)_p &\simeq Sp(2)_p\text{ for all } p.
\end{align*} 
Let $p$ be a prime number, then
$$ (G_{2, b})_p \simeq (G_2)_p
\quad\text{or}\quad
(G_{2, b})_p \simeq S^3_p \times S^{11}_p$$
depending on $p$ by \cite{mnt}. 
Therefore if $X$ is a 1--connected finite $H$--space of rank 2, then $X_p$ is
homotopy equivalent to $G_p$ or 
$S^m_p \times S^n_p$ for each prime number $p$, where $G$ is a 
1--connected Lie group of
rank 2 and $m, n \in \{3, 5, 7, 11\}$. 
It is easy to see that $t(S^m_p \times S^n_p) \leq 2$ for any prime $p$. 
On the other
hand, ${\mathcal Z}(G_p) = {\mathcal Z}(G)_p$ by \fullref{localization} 
and
${\mathcal Z}(G)_p \subset {\mathcal Z}(G)$ since ${\mathcal Z}(G)$ is 
a finite nilpotent group for a 1--connected compact Lie group $G$ of rank 2.
Thus $t(G_p) \leq 2$ for an arbitrary prime number $p$. Thus we obtain 
that  
if $X$ is a 1--connected finite $H$--space of rank 2, then $t(X_p) \leq 2$ 
for all
prime numbers $p$ and hence 
$t(X) \leq 2$ by \fullref{upperlemma}.

Next we will
show that $t_\infty(X) > 1$ for our spaces. Namely, 
${\mathcal Z}^\infty(X) \not= 0$. 
First we consider the case where $X$ = $S^7 \times S^7$.
As in the proof of
\fullref{rank2}  
${\mathcal Z}^\infty(S^7 \times S^7) = \pi_{14}(S^7) \oplus \pi_{14}(S^7)$ 
which is
isomorphic to $\mathbb{Z}_{120} \oplus \mathbb{Z}_{120}$ by Toda \cite{toda}.
We should note that $[S^7 \times S^7, 
S^7 \times S^7]$ is a group despite that $S^7$ is not homotopy associative
(see Mimura--{\=O}shima \cite{mio}) though we do not need the group 
structure for our 
purpose.  
Let $X$ be an $H$--space. If $n \geq \dim~X$,  by a result of
James \cite{ja} there exists a bijection
\begin{equation}\label{T}
T \co {\mathcal Z}^n(X) \to {\mathcal E}_\#^n(X) 
\end{equation}  
defined by $f \to 1 + f$. Here ${\mathcal E}_\#^n(X)$
is the group of homotopy classes of self-homotopy equivalences  which 
induce the
identity map on $\pi_i(X)$ for $i \leq n$. 
$$ {\mathcal E}_\#^\infty(E_k)_5 \cong {\mathcal E}_\#^\infty((E_k)_5),$$  
by \cite{ma}. Note that $(E_k)_5$ is homotopy equivalent to $S_5^3 \times 
S_5^7.$
Namely $E_k$ is 5-regular. The group ${\mathcal E}_\#^\infty(S_5 ^3
\times S_5^7)$ is easily shown to be isomorphic to $\mathbb{Z}_{5}$ (see
\cite{ma2} or \cite{mio}).
Hence 
$${\mathcal E}_\#^\infty(E_k)_5 \cong \mathbb{Z}_{5}.$$ 
Thus by the bijection $T$, ${\mathcal Z}^\infty(E_k)$ is not trivial for 
$k$ = 0, 1, 3, 4, 5.
Similarly it is known that the spaces $G_{2, b}$ ($-2 \leq b \leq 5$) are 
7--regular,
that is 
$$(G_{2, b})_7 \simeq S^3_7 \times S^{11}_7.$$
Therefore we obtain 
$${\mathcal E}_\#^\infty(G_{2, b})_7 \cong {\mathcal E}_\#^\infty(S^3 
\times S^{11})_7.$$
The group  ${\mathcal E}_\#^\infty(S^3 \times S^{11})_7$ is not trivial 
($\cong \mathbb{Z}_7$ see \cite{ma2}). Thus
${\mathcal Z}^\infty(G_{2, b})$ is non-trivial by the same reason for 
$E_k$. 

Consequently, $t_\infty(X) > 1$ for all the 1--connected finite $H$--spaces 
of rank 2.
We complete the proof. 
\end{proof} 

\begin{rem} In the above proof we cannot use \fullref{localization} 
directly 
to show that $t_\infty(X) > 1$  because the spaces are not
necessarily homotopy associative.  
\end{rem}

\section[SU(4) and Sp(3)]{$SU(4)$ and $Sp(3)$} \label{sec:sec3}
In this section we consider rank 3 Lie groups $SU(4)$ and $Sp(3)$.
The statement of our theorem 
is completely the same as that of \fullref{rank2}, but its proof is more 
complicated.
\begin{theorem}\label{th.SU(4)}
$t(G) = t_\infty(G) = 2$
for $G = SU(4)$ and $Sp(3)$.
\end{theorem} 
Our arguments in this section depend heavily on Oka's results in 
\cite{oka1}.

Let  $C_\varphi$ be the mapping cone of
$\varphi \co X \to Y$, $q \co C_\varphi \to \Sigma X$ the projection map.
Recall that there exists an action of
$[\Sigma X, C_\varphi]$ on 
$[C_\varphi, C_\varphi]$ induced by the
coaction map $\ell \co C_\varphi \to \Sigma X \vee C_\varphi$.
Namely, for  $\alpha \in [\Sigma X, C_\varphi]$ and $g \in [C_\varphi, 
C_\varphi]$,
$\alpha \cdot g$ is the following composition:
$$ C_\varphi \xrightarrow{\ell} \Sigma X \vee C_\varphi
\xrightarrow{\alpha \vee g} C_\varphi \vee C_\varphi \xrightarrow{\nabla}
C_\varphi$$
where $\nabla$ is the folding map. The following lemma is well known. 
\begin{lemma}[Hilton {\cite[Theorem 15.7]{hi}}]
\label{actionlemma}
Let 
$\alpha \in  [\Sigma X, C_\varphi]$ and  $g , h \in
[C_\varphi, C_\varphi]$ be arbitrary elements.     
If $C_\varphi$ is an $H$--space, then
$$\alpha \cdot g = q^*(\alpha) + g \text{ and }   
h \circ (q^*(\alpha) + g) = h \circ q^*(\alpha) + h \circ g,$$
where $+$ denotes the addition induced by the $H$--structure of $C_\varphi$. 
\end{lemma}


Now we consider the case where $G = SU(4)$.
As noted by Oka \cite[(2.2)]{oka1}, there exists
a homotopy equivalence as follows. 
$$SU(4)/SU(4)^7 \to \Sigma K \vee S^{15},$$
where $SU(4)^7$ is the 7--skeleton of $SU(4)$ and $K = (S^7 \vee S^9) \cup 
e^{11}$. We denote by
$$ \pi_1 \co SU(4) \to \Sigma K$$
the  projection map. We have the following maps (homomorphisms):
\begin{align*}
&q_{SU(4)}^* \co \pi_{15}(SU(4)) \to [SU(4), SU(4)]\\
&\pi_1^* \co [\Sigma K, SU(4)] \to [SU(4), SU(4)]
\end{align*}
\begin{lemma}\label{Imagelemma}
${\mathcal Z}(SU(4))$ and ${\mathcal Z}^\infty(SU(4))$ are generated by 
elements of
${\rm Im} \;q_{SU(4)}^* \cup {\rm Im} \pi_1^*$. In particular they are 
abelian groups.
\end{lemma}
\begin{proof}
We show our claim is true for  ${\mathcal Z}(SU(4))$ since ${\mathcal 
Z}^\infty(SU(4))$ is a subgroup of ${\mathcal Z}(SU(4))$. 
Let ${\mathcal E}_*(X)$
denote the the group of homotopy classes of self-homotopy equivalences  
which induce the identity map on the integral homology groups of $X$.
Since $H^*(SU(4))$ is isomorphic to the exterior algebra
$\Lambda_{\mathbb{Z}} (x_3, x_5, x_7)$, we have 
$$ {\mathcal E}_\#^n(SU(4)) \subset {\mathcal E}_*(SU(4)).$$
Here $n \geq 7$.
By \cite[Theorem 2.4, Thorem 8.3]{oka1},  ${\mathcal E}_*(SU(4))$ is 
generated by elements
$$q_{SU(4)}^*(x) + 1_{SU(4)} \quad\text{and}\quad
\pi_1^*(y) + 1_{SU(4)},$$
where $x \in \pi_{15}(SU(4))$ and  $y \in [\Sigma K, SU(4)]$. We easily
see that
\begin{align*}
(q_{SU(4)}^*(x) + 1_{SU(4)})^n &=  q_{SU(4)}^*(nx) + 1_{SU(4)}, \\
(\pi_1^*(y) + 1_{SU(4)})^n &= \pi_1^*(ny) + 1_{SU(4)}
\end{align*}
for $n \in \mathbb{Z}$, (cf \cite{oka1}). 
Moreover,   
$$(\pi_1^*(y) + 1_{SU(4)})\circ(q_{SU(4)}^*(x) + 1_{SU(4)}) = 
q_{SU(4)}^*(y \circ \pi_1 \circ x + x) + \pi_1^*(y) + 1_{SU(4)},$$
by \fullref{actionlemma}, and 
$$(q_{SU(4)}^*(x) + 1_{SU(4)})\circ(\pi_1^*(y) + 1_{SU(4)}) = 
q_{SU(4)}^*(x) + \pi_1^*(y) + 1_{SU(4)}.$$
The second equality follows from
$$q_{SU(4)}\circ(\pi_1^*(y) + 1_{SU(4)}) = q_{SU(4)}.$$
Therefore by the bijection $T$ given in \eqref{T} ${\mathcal Z}(SU(4))$ 
is
generated by some elements of ${\rm Im}\; q_{SU(4)}^*$ and ${\rm Im} \;
\pi_1^*$.

It is known that ${\rm Im} \; q_{SU(4)}^*$ is in the center of $[SU(4), 
SU(4)]$ by
\cite{wh}. Therefore ${\mathcal Z}(SU(4))$ is an 
abelian group since ${\rm Im} \; \pi_1^*$ is abelian.    
\end{proof}

Now we prove the main theorem in this section.
\begin{proof}
\rm{(of \fullref{th.SU(4)})} \qua By \fullref{Imagelemma} an element of
${\mathcal Z}(SU(4))$ is of the form $z = f + g$ with $f \in {\rm 
Im} \; q_{SU(4)}^*$
and $g \in {\rm Im} \; \pi_1^*$. Since $f$ induces the trivial map on 
$\pi_i(SU(4))$, $i \leq 15$, $g$ is an element of  ${\mathcal Z}(SU(4))$. 
Therefore $g{\circ}f{=}0$.  We also have  $f{\circ}g{=}0$ easily.   Let 
$f_1, f_2$ 
be elements of ${\rm Im} \; q_{SU(4)}^*$, then 
$f_1{\circ}f_2{=}0$.
Now $[\Sigma K, SU(4)]$ is a finite group and $\pi_{1*} \co [\Sigma K, 
SU(4)] \to [\Sigma K, \Sigma K]$ is a 
homomorphism. We have the following isomorphism by \cite[Lemma 3.3]{oka1}:
$$ [\Sigma K, \Sigma K] \cong \mathbb{Z} \oplus \mathbb{Z} \oplus \mathbb{Z}.$$
Therefore $g_1{\circ}g_2 = 0$ for $g_1,  g_2$ $\in {\rm Im} \; \pi_1^*$. 
Let $h \co SU(4) \to SU(4)$ be any map, then we have $(f + g)\circ h 
= f \circ h + g \circ h,$ since the addition is defined by the group 
structure of $SU(4)$.
Moreover by \fullref{actionlemma} $h \circ (f + g) = h \circ f + h \circ 
g$.
Here $f \in {\rm Im} \; q_{SU(4)}^*$
and $g \in {\rm Im} \; \pi_1^*$ as above. 
Consequently, the composition of any two elements
$z = f_1 + g_1$, $z' = f_2 + g_2$ of ${\mathcal Z}(SU(4))$ is trivial, 
where
$f_1, f_2 \in {\rm Im}\;q_{SU(4)}^*$
and  $g_1, g_2 \in {\rm Im}\; \pi_1^*$. Therefore we have obtained that 
$t(SU(4)) \leq 2$ and thus  $t_\infty(SU(4)) \leq 2$ 
(recall that $t_\infty(SU(4)) \leq t(SU(4))$).

Next we will show that 
$1 < t_\infty(SU(4))$, that is, ${\mathcal Z}^\infty(SU(4))$ is not 
trivial. It is known that
$SU(4)_3$ 
is homotopy equivalent to $Sp(2)_3 \times S^5_3$. Therefore ${\mathcal 
Z}^\infty (Sp(2)_3)$ $\subset$
${\mathcal Z}^\infty(SU(4)_3)$. As was mentioned in the proof of 
\fullref{rank2},
${\mathcal Z}^\infty(Sp(2))$
is isomorphic to $\mathbb{Z}_{120}$, and  hence 
${\mathcal Z}^\infty(SU(4)_3)$ is nontrivial. 
Therefore by \fullref{localization}, ${\mathcal Z}^\infty(SU(4))$ is also 
nontrivial.   

The proof for $Sp(3)$ is parallel to that of $SU(4)$ by using 
\cite[Theorem 2.5, Theorem 4.3]{oka1}. We can show that 
$t(Sp(3)) \leq 2$ as in the 
$SU(4)$ case.
To show that $t(Sp(3)) = 2$, we use the equivalence 
$ Sp(3)_7 \simeq (S^3 \times S^7 \times S^{11})_7$ and the nontriviality 
of ${\mathcal Z}^\infty((S^3 \times S^7 \times S^{11})_7)$.  
This nontriviality is obtained by the existence of an essential map:
$$ S^3 \times S^7 \times S^{11} \to S^3 \times S^{11} \to S^{14} 
\xrightarrow{\alpha_1} S^3 \to S^3 \times S^7 \times S^{11}$$
where $\alpha_1$ is a generator of $\pi_{14}(S^3)_7 \cong \mathbb{Z}_7$ and 
other maps are
the canonical projections and
the inclusions.
\end{proof}
\section{The lower bounds for classical groups} \label{sec:sec4}
In this section we give a lower bound for $t_*(G)$ when $G$ is
$SU(n)$ or $Sp(n)$. We should admit that it is a crude one,
but it gives us the theorem which states that $t_\infty(G)$ could be 
arbitrarily large for
classical groups.
\begin{prop}\label{nilpropo}
Let $X$ be a homotopy associative finite $H$--space, then $t(X) \geq t(X_0)$ 
and 
$t_\infty(X) \geq t_\infty(X_0)$.
\end{prop}
\begin{proof}
The rational cohomology ring $H^*(X; \mathbb{Q})$ is isomorphic to 
the exterior algebra $\Lambda_{\mathbb{Q}}(x_1, \cdots x_r)$ on primitive 
elements $x_i$ with $\dim x_i = n_i$ odd.
Let $n$ denote $t(X)$ and $a_1, \dots, a_n$ be elements of ${\mathcal 
Z}(X_0).$
We will show that $a_1 \circ \cdots \circ a_n = 0$. As $X_0$ is homotopy 
equivalent to
the product of the Eilenberg-MacLane spaces, the elements of $[X_0, X_0]$ 
are determined by
their induced maps on cohomology groups.  Actually ${\mathcal Z}(X_0)$ is 
isomorphic to the module generated by the decomposable elements of degree 
$\{\dim x_i\}$. 
We define a basis for ${\mathcal Z}(X_0)$ as follows. Let 
$\{x_{i_1}x_{i_2}\cdots x_{i_j}\}$ 
with $i_1 < i_2 <\cdots < i_j$ be the basis for the module of 
decomposable elements of
$H^*(X; \mathbb{Q}).$
Let $\{{i_1} <  {i_2} < \cdots <  {i_j}\}$ be a set of
subsets of $\{1, \dots, r\}$ such that
$$\dim x_{i_1} + \dim x_{i_2} + \cdots + \dim x_{i_j} = \dim x_i.$$
Let  $\wedge \co\smash{\prod_{k=1}^j}K(\mathbb{Q}, n_{i_k})
\to \smash{\wedge_{k=1}^j}K(\mathbb{Q}, n_{i_k})$ be the projection to the smash 
product.
Then the map $f_{{i_1}{i_2} \cdots {i_j}}$ is defined by the composition
$$ X_0 \to
{\textstyle \prod_{k=1}^j}K(\mathbb{Q}, n_{i_k}) \xrightarrow{\wedge}
{\textstyle \wedge_{k=1}^j} K(\mathbb{Q}, n_{i_k}) \to K(\mathbb{Q}, n_i)
\to X_0,$$
where the first and the
last maps in the composition are the projection and inclusion maps, the 
third map is the map corresponding to the cohomology element 
$x_{i_1}x_{i_2}\cdots x_{i_j}$.   
Now we assume that $a_1 \circ \cdots \circ a_n$ is not trivial. Hence
$(a_1 \circ \cdots \circ a_n)^*(x_k) \not= 0$ for some $x_k$.   
We have  
$$(a_1 \circ \cdots \circ a_{i-1})^*(x_k) = \Sigma t_j 
x_{j_1}x_{j_2}\cdots x_{j_\ell} ,
$$
for $i \leq n$, where $t_j$ are nonzero rational numbers.  
Thus 
$$ a_i^*(x_{j_1}x_{j_2}\cdots x_{j_\ell}) $$
is nontrivial for some $x_{j_1}x_{j_2}\cdots x_{j_\ell}$. It follows that 
for each $x_{j_t}$ there exist decomposable elements $x_{s_1}x_{s_2}\cdots 
x_{s_k}$ such that $\dim x_{j_{t}} = \dim x_{s_{1}}x_{s_{2}}\cdots 
x_{s_{k}}.$ Therefore the maps $f_{s_{1}s_{2}\cdots s_{k}}$ are defined 
for $\{x_{s_{1}}x_{s_{2}} \cdots x_{s_{k}}\},$ and  $ 
a_i^*(x_{j_1}x_{j_2}\cdots x_{j_\ell})$ is $(\Sigma r_{s_{1}s_{2}\cdots 
s_{k}}f_{s_{1}s_{2}\cdots s_{k}})^*(x_{j_1}x_{j_2}\cdots x_{j_\ell})$, 
where $r_{s_{1}s_{2}\cdots s_{k}}$ are rational numbers.   
From the nontriviality of $(a_1 \circ \cdots \circ a_n)^*(x_k)$, we obtain 
the nontrivial iterated composition  :  
$$ g_1 \circ \cdots \circ g_n \in {\mathcal Z}({X_0})$$
such that $g_i = \Sigma f_{{i_1}{i_2} \cdots {i_j}}$ for each $g_i$.  Let 
$m_1, \ldots, m_n$ be nonzero integers.
Since we see that 
$$ (m_1g_1 \circ \cdots \circ m_ng_n)^*(x_k) = dx_{i_1}x_{i_2}\cdots 
x_{i_j}$$
for some $x_{i_1}x_{i_2}\cdots x_{i_j}$ with ${i_1} <  {i_2} < \cdots <  
{i_j}$ and nontrivial integer $d$,  thus the composition $m_1g_1 \circ 
\cdots \circ m_ng_n$ is
essential.  Here we should note that $(m_1g_1 \circ \cdots \circ 
m_ng_n)^*(x_k)$ is not equal to $m_1m_2\cdots m_n (g_1 \circ \cdots \circ 
g_n)^*(x_k)$ in general.  
As the homomorphism ${\mathcal Z}(X) \to {\mathcal Z}(X_0)$ is the 
localization by
\fullref{localization}, $m_i g_i$ is 
an element of ${\mathcal Z}(X)$ for some nonzero integer $m_i$. So, we can 
find a nontrivial
composition 
$m_1g_1 \circ \cdots \circ m_ng_n$ with $m_i g_i$ $\in {\mathcal Z}(X)$, 
this is a contradiction, hence we obtain that $t(X) \geq t(X_0)$.  

Since ${\mathcal Z}(X_0) = {\mathcal Z}^\infty(X_0)$ in our case, we can 
show that $t_\infty(X) \geq t_\infty(X_0)$ similarly.      
\end{proof}

Now we apply \fullref{nilpropo} to the classical groups $SU(n)$ and 
$Sp(n)$.
\begin{theorem} Let $\ell$ be a natural number. Then
\begin{eqnarray*}
t_\infty(SU(n)) > \ell& \text{ for } & n  \geq (3^{\ell} + 1)^2/2\\ 
t_\infty(Sp(n)) > \ell& \text{ for } & n \geq   (2\cdot 5^{2\ell} + 5^\ell + 1)/4.
\end{eqnarray*} 
\end{theorem}
\begin{proof} 
Recall that $SU(n)$ is rationally equivalent to 
$$ S_0^3 \times \cdots \times S_0^{2n-1}.$$ By \fullref{nilpropo} it 
suffices to construct
a desired nontrivial
composition in ${\mathcal Z}^\infty(S_0^3 \times \cdots \times S_0^{2n 
-1})$.
To this end, we take the smash product for each
successive $3^k$ spheres in the product space. 
We already have dealt with such a map in \fullref{nilpropo}, that is $f_i$.
However here we need a more careful consideration about dimensions. Now 
we assume that $n$ is sufficiently large.   
We let 
$$ \wedge_{2i-1,2i+1,2i+3} \co S_0^{2i-1} \times S_0^{2i+1} \times 
S_0^{2i+3} \to S_0^{6i +3}.$$
denote the projection map to the smash product. Then we take the product 
of these maps
$$ \prod_{i=1}^{3^{\ell-1}} \wedge_{6i-3,6i-1,6i+1} \co 
\prod_{i=1}^{3^\ell}S_0^{2i+1} \to
\prod_{i=1}^{3^{\ell-1}}S_0^{18i - 3}.$$
We define a map $a_1 \co S_0^3 \times \cdots \times S_0^{2n-1} \to S_0^3 
\times \cdots \times S_0^{2n-1}$ to
be the following composition.
$$\prod_{i=1}^{n-1} S_0^{2i+1} \to 
\prod_{i=1}^{3^\ell}S_0^{2i+1} \xrightarrow{\Pi_{i=1}^{3^{\ell-1}}
\wedge_{6i-3,6i-1,6i+1} }
\prod_{i=1}^{3^{\ell-1}}S_0^{18i - 3} \to \prod_{i=1}^{n-1} S_0^{2i+1}$$ 
where the first map is the projection and the third map is the inclusion.

Similarly we construct $a_2$ as follows:
$$\prod_{i=1}^{n-1} S_0^{2i+1} \to 
\prod_{i=1}^{3^{\ell-1}}S_0^{18i-3} \xrightarrow{\Pi_{i=1}^{3^{\ell - 2}} 
\wedge_{54i-39,54i
-21, 54i-3} }
\prod_{i=1}^{3^{\ell - 2}}S_0^{162i - 63} \to \prod_{i=1}^{n - 1} 
S_0^{2i+1}.$$

We continue this process and finally we obtain $a_\ell$:
$$\prod_{i=1}^{n-1} S_0^{2i+1} \to 
\prod_{i=1}^{3}S_0^{3^{\ell - 1}((2i - 1)3^{\ell -1} + 2 )}
\to
S_0^{3^{2\ell} + 2\cdot 3^\ell} \to
\prod_{i=1}^{n-1} S_0^{2i+1}.$$ 
Clearly a map $a_\ell \circ a_{\ell -1} \cdots \circ a_2 \circ a_1$ 
induces a nontrivial
map on cohomology and moreover induces the trivial map on the homotopy 
groups. The construction
is possible if $2n -1 \geq 3^{2\ell} + 2 \cdot3^{\ell}$. Namely, if 
$n  \geq (3^{\ell} + 1)^2/2 $
then 
$$ t_\infty(SU(n)) > \ell.$$
This completes the claim for $SU(n)$. 
To prove the $Sp(n)$ case we can use the same methods for the $SU(n)$ 
case. For $Sp(n)$ this time
we consider the projection maps to the smash products from successive 
$5^k$ spheres instead 
of $3^k$ spheres
which were used in the proof of $SU(n)$. Then we obtain that if 
$n \geq   (2\cdot 5^{2\ell} + 5^\ell + 1)/4$,  
$$t_\infty(Sp(n)) > \ell.\proved$$
\end{proof}

\begin{rem} We can apply the similar arguments in the proof of the
above theorem to other classical groups. We obtain
$$t_\infty(U(n)) > \ell\hspace{0.2cm} \text{for}\hspace{0.2cm}  n  \geq 
(3^{\ell} + 1)^2/2,$$
and  
\begin{gather*}
t_\infty(G(m)) > \ell \text{ for} 
\begin{cases}
m  \geq (2\cdot 5^{2\ell} + 5^\ell + 3)/2 &  \text{if $m$ is odd}\\ 
m \geq (2\cdot 5^{2\ell} + 5^\ell + 5)/2 &  \text {if $m$ is even,}
\end{cases}
\end{gather*}
where $\ell$ is an integer, and $G(m) = SO(m), {\rm Spin}(m)$ or $O(m)$.     
\end{rem}

\bibliographystyle{gtart}
\bibliography{link}

\end{document}